\documentclass{amsart}
\usepackage{latexsym,amsfonts,amssymb,amsthm,amsmath}
\usepackage[all]{xy}
\usepackage{todonotes}
\usepackage{multicol}
\usepackage[hidelinks]{hyperref}
\usepackage{thmtools, thm-restate}

\newtheorem{theorem}{Theorem}[section]
\newtheorem{lemma}[theorem]{Lemma}

\theoremstyle{definition}
\newtheorem{definition}[theorem]{Definition}
\newtheorem{example}[theorem]{Example}
\newtheorem{proposition}[theorem]{Proposition}

\newtheorem*{maintheorem}{Main Theorem}

\theoremstyle{remark}

\numberwithin{equation}{section}

\newcommand{\ds}{\displaystyle}
\newcommand{\Z}{\mathbb{Z}}

\newcommand{\mf}[1]{\underline{#1}}
\newcommand{\B}[2]{\underline{B}_{#1,#2}}
\newcommand{\dcurve}[1]{\ar@/_1pc/[d]_{#1}}
\newcommand{\ucurve}[1]{\ar@/_1pc/[u]_{#1}}

\begin{document}

\title{The slices of $S^n \wedge H\mf{\Z}$ for cyclic
    $p$-groups}

\author{Carolyn Yarnall}
\address{Department of Mathematics, University of Kentucky, Lexington, KY 40506}





\begin{abstract}
The slice filtration is a filtration of equivariant spectra. While the tower is analogous to the
Postnikov tower in the nonequivariant setting, complete slice towers
are known for relatively few $G$-spectra. In this paper, we determine the slice tower for all $G$-spectra of the form $S^n \wedge H\mf{\Z}$ where $n\geq 0$ and $G$ is a cyclic $p$-group for $p$ an odd prime.
\end{abstract}

\maketitle

\section{Introduction}

The slice filtration is a filtration of equivariant spectra developed
by Hill, Hopkins, and Ravenel in their solution to the Kervaire
invariant one problem \cite{HHR}. It is a generalization of Dugger's
$C_2$-equivariant filtration \cite{Dugger} and is modeled on the
motivic filtration of
Voevodsky \cite{Voevodsky}. While it is an equivariant analogue of the Postnikov tower,
there are some marked differences. These differences arise from the fact that we use representation spheres rather than ordinary spheres in the construction. If $V$ is a $G$-representation, a representation sphere $S^V$ is the one-point compactification of the associated representation space. 

\begin{definition}
For each integer $n$, let $\tau_{\geq n}$ denote the localizing subcategory of $G$-spectra of the form $G_+ \wedge_H S^{m\rho_H-\epsilon}$ where $H$ ranges over all subgroups of $G$, $\rho_H$ denotes the regular representation of $H$, $\epsilon = 0,1$, and the dimension of the underlying sphere is $\geq n$.
\end{definition}

We think of a localizing subcategory as the category of acyclics for a localization functor on $G$-spectra. More precisely we mean a full subcategory that is closed under weak equivalences, cofibers, extensions, retracts, wedges, and directed colimits. The localizing subcategory $\tau_{\geq n}$ is the collection of slice cells used to construct the slice tower.

\begin{definition}
Let $X$ be a $G$-spectrum. The \emph{slice tower} of $X$ is the tower consisting of the following data:
\begin{enumerate}
\item Stages $P^n(X)$ where $P^n(-)$ denotes the localization functor associated to $\tau_{\geq n}$.
\item Natural maps $P^n(X) \rightarrow P^{n-1}(X)$ with fiber denoted
  $P^n_n(X)$ and referred to as the \emph{$n$th slice} of $X$.
\end{enumerate}
\end{definition}

These definitions are given in \cite{ESP} and are equivalent to the original definitions from \cite{HHR}.
From them we see that the construction resembles that of the Postnikov tower but differences arise due to the nature of spectra in $\tau_{\geq n}$ used. Most notably, these spectra are often less connected than ordinary spheres. For example, the $C_2$-representation sphere $S^{\rho_2 - 1}$ is a slice-cell in $\tau_{\geq 0}$. However, the map $S^0 \hookrightarrow S^{\rho_2 - 1}$ is the inclusion of fixed points and thus the homotopy in dimension $0$ is not trivial. In general, objects in $\tau_{\geq n}$ are often less than $n$ connected and hence lower homotopy groups may be affected in forming each stage $P^n(X)$.
Thus, unlike the fibers from the Postnikov tower, slices
need not be Eilenberg-MacLane spectra.

As a result, the slice tower can be difficult to compute and thus far has been determined for relatively few
genuine $G$-spectra. Hill, Hopkins, and Ravenel have
determined the complete slice filtration for
$MU_{(C_{2^k})}$ \cite{HHR}. The slice tower for certain
Eilenberg-MacLane spectra was given by Hill in \cite{ESP} and for more
general Eilenberg-MacLane spectra by Ullman in
\cite{Ullman}. As the slice tower does not commute with all 
suspensions, knowing the towers for $H\mf{M}$ does not immediately
give us the towers for $S^V \wedge H\mf{M}$ for representations
$V$ or even $S^n \wedge H\mf{\Z}$. On this front, Hill, Hopkins, and Ravenel have
computed the slice tower for $S^V \wedge H\mf{\Z}$ where $V$ is a multiple of a particular nontrivial subrepresentation of $\rho_{C_{p^k}}$ \cite{RO(G) slices}. What we present here is a different
flavor of the work in \cite{RO(G) slices}, namely, the positive
integer-graded suspensions of $H\mf{\Z}$ for cyclic $p$-groups where $p$ is an odd prime.

\begin{maintheorem}
The slice tower for $S^n \wedge H\mf{\Z}$ for $G = C_{p^k}$ where $p$ is an odd prime consists of the following data:
\begin{itemize}
\item The slice sections $P^m(S^n \wedge H\mf{\Z})$ are of the form $S^W \wedge H\mf{\Z}$ where $W$ is a $C_{p^k}$-representation of dimension $n$ and $n \leq m \leq (n-2)p^k - 1$.
\item The nontrivial slices $P^m_m(S^n \wedge H\mf{\Z})$ are of the form $S^V \wedge H\B{i}{j}$ where $\B{i}{j}$ is a $C_{p^k}$-Mackey functor, $V$ is a $C_{p^k}$-representation of dimension $m$, and $m$ takes on select values in the range $n \leq m \leq (n-2)p^k -1$. All other slices are trivial.
\end{itemize}
\end{maintheorem}

The exact objects that appear in this theorem will be defined along the way to the final statement of the theorem in Section 4. In Section 2  we provide a definition of the Mackey functors $\B{i}{j}$ appearing in the slices of our tower. In Section 3 we present the collection of spectra of the form $S^V \wedge H\B{i}{j}$ from the statement of our Main Theorem and the fiber sequences that comprise our tower. After a more precise statement of the main theorem in Section 4, we provide a few examples. The final section consists of proofs of the essential theorems from Section 3.

\section{Background: Mackey Functors}

The definitions in this section are given by Hill, Hopkins, and Ravenel in \cite{RO(G) slices}. We note that the Weyl action on these Mackey functors is trivial and thus we omit it from all definitions and discussions in this section.

First, we note that we will depict all Mackey functors using the diagrams of Lewis from \cite{Lewis}. Thus, if $\mf{M}$ is a $C_p$-Mackey functor, we will represent it as
\[\xymatrix{ \mf{M}(C_p/C_p) \ar@(l,l)[d]_{res} \\
              \mf{M}(C_p/e) \ar@(r,r)[u]_{tr}  } \] 

For the constant $C_{p^k}$-Mackey functor $\mf{\Z}$, all restriction maps are the identity and transfer maps are multiplication by $p$. The restriction and transfer maps roles are reversed for its dual, $\mf{\Z}^*$. We can make strategic exchanges of the restriction and transfer maps to obtain similar Macky functors.

\begin{definition}\label{Z(i,j)}
Let $G=C_{p^k}$ and $0 \leq j < i \leq k$ be integers. For each pair, let $\mf{\Z}(i,j)$ denote the Mackey functor with constant value $\Z$ for which the restriction and transfer maps are
\begin{center}
\begin{multicols}{2}

$res_{C_{p^m}}^{C_{p^{m+1}}} = \begin{cases} 1 & m < j \\ p & j \leq m < i \\ 1 & i \leq m \end{cases}$

$tr_{C_{p^m}}^{C_{p^{m+1}}} = \begin{cases} p & m<j \\ 1 & j \leq m < i \\ p & i \leq m \end{cases}$
\end{multicols}
\end{center}
\end{definition}

Notice that the restriction of $\mf{\Z}(i,j)$ to $C_{p^j}$ is the constant Mackey functor $\mf{\Z}$. If we allow $i = j$ then $\mf{\Z} = \mf{\Z}(i,i)$. We also note that the dual to the constant Mackey functor $\mf{\Z}^*$ is equivalent to $\mf{\Z}(k,0)$.

\begin{example}
  For $G = C_{p^2}$, we have the following $\mf{\Z}(i,j)$:
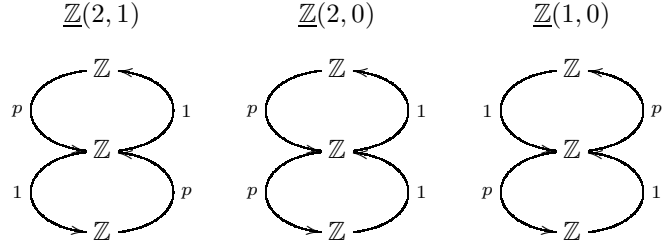
\[\xymatrix@R=.5pc { \mf{\Z}(2,1)  & & \mf{\Z}(2,0) & & \mf{\Z}(1,0) \\
              \Z \ar@(l,l)[dd]_p & & \Z \ar@(l,l)[dd]_p & & \Z \ar@(l,l)[dd]_1 \\
\\
              \Z \ar@(l,l)[dd]_1 \ar@(r,r)[uu]_1 & & \Z \ar@(l,l)[dd]_p \ar@(r,r)[uu]_1 & & \Z \ar@(l,l)[dd]_p \ar@(r,r)[uu]_p \\
\\
              \Z \ar@(r,r)[uu]_p & & \Z \ar@(r,r)[uu]_1 & & \Z \ar@(r,r)[uu]_1 } \]
\end{example}

We can define maps from $\mf{\Z}(i,j)$ to $\mf{\Z}$ by choosing the image of the element $1$ in $\mf{\Z}(i,j)(G/
e)$. From such maps we obtain the following Mackey functors. 

\begin{definition}\label{B(i,j)}
  Let $G = C_{p^k}$ and $1 \leq j < i \leq k$ be integers such that $i+j \leq k$. Let $\B{i}{j}$ denote the quotient Mackey functor associated to the inclusion $\mf{\Z}(i+j,j) \rightarrow \mf{\Z}$. That is, the $C_{p^k}$-Mackey functor $\B{i}{j}$  is defined as follows: $$ \B{i}{j}(C_{p^k}/C_{p^m}) = \begin{cases} \Z/p^{i} & m \geq i + j \\
  \Z/p^{m-j} & j < m < i + j \\
  0 & m \leq j \end{cases} $$  
where all restriction maps are the canonical quotients $q$ and transfer maps are multiplication by p. \end{definition}

\begin{example}
  For $G = C_{p^2}$, we have the following $\B{i}{j}$:
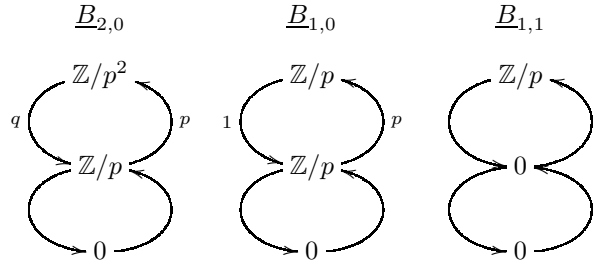
\[\xymatrix @R=.5pc{ \B{2}{0} & & \B{1}{0} & & \B{1}{1}\\
              \Z/p^2 \ar@(l,l)[dd]_q & & \Z/p \ar@(l,l)[dd]_1 & & \Z/p \ar@(l,l)[dd] \\
\\
              \Z/p \ar@(l,l)[dd] \ar@(r,r)[uu]_p & & \Z/p \ar@(l,l)[dd] \ar@(r,r)[uu]_p & & 0 \ar@(l,l)[dd] \ar@(r,r)[uu] \\
\\
              0 \ar@(r,r)[uu] & & 0 \ar@(r,r)[uu] & &  0 \ar@(r,r)[uu] } \]
\end{example}

\section{Slice Tower Objects}
In order to show that a tower is the slice tower for a
$G$-spectrum $X$, by Proposition 4.45 in \cite{HHR} it is sufficient
to show that the limit of the tower is $X$, the colimit of the tower
is contractible, and the $m$th fiber is an $m$-slice for all $m$. The
towers that we present for $S^n \wedge H\mf{\Z}$ are finite so the
 interesting work amounts to determinining the fiber sequences
that make up the tower and showing that each of these fibers is $m$-slice.

\subsection{Slices}

We will now define a collection of spectra that are $m$-slice for certain values of $m$. In Section 4 we will see that these are precisely the nontrivial slices in the tower for $S^n \wedge H\mf{\Z}$. For the remainder of this section, we assume $n\geq 3$.  

Our spectra will all be of the form $S^V \wedge H\mf{M}$ where $V$ is a sum of subrepresentations of the regular represenation $\rho_G$. When $G = C_{p^k}$ we will often write $\rho_G$ as $\rho_{p^k}$. This representation has the decomposition:
\[ \rho_{p^k} = 1 + \bigoplus_{j=1}^{\frac{p^k-1}{2}} \lambda(j)\]
where $\lambda(j)$ is the composition of the inclusion of the $p^k$th roots of
unity with a degree $j$ map on $S^1$. All objects in our work are considered $p$-locally and thus
$S^{\lambda(i)}$ and $S^{\lambda(j)}$ are equivalent whenever the $p$-adic orders of $i$ and $j$ agree. Thus, we need only keep track of
$\lambda(p^j) = \lambda_j$. Working this way, we have an equivalent
decomposition of $\rho_{p^k}$:
\[ \rho_{p^k} = 1 + \frac{p-1}{2} \lambda_{k-1} +
\frac{p(p-1)}{2}\lambda_{k-2} + \cdots + \frac{p^{k-1}(p-1)}{2}\lambda_0\]

Let $d$ denote the number of integers of the same parity as $n$ that lie between $\frac{n}{p}$ and $n-2$, inclusive. Via a simple counting argument we obtain a useful way of writing $d$: \[d = \frac{1}{2}\left(n-\Big(\frac{n-n_0}{p}\Big)-\delta\right)\] where $n_0$ is the residue of $n$ modulo $p$ and \[ \delta = \begin{cases} 2 & \mbox{if } n_0 \mbox{ is even}\\ 
                              1 & \mbox{if } n_0 \mbox{ is odd}\\
                              0 & \mbox{if } n_0 = 0 \end{cases}\]
We label the integers between $\frac{n}{p}$ and $n-2$ that are of the same parity as $n$, $m_1 < m_2 < \dots < m_d$ so $m_i = n - 2d + 2i - 2$.

\begin{definition}
\label{V(a,b)}
  Let $a$, $b$ be integers such that $1 \leq a \leq k$, $1 \leq b \leq d$. We use $V_{(a,b)}(n)$ to denote the $G$-representation:
\[V_{(a,b)}(n) = (n-2)\rho_G - 1 - \bigoplus_{i=1}^{\ell} \lambda(i)\]
where \[ \ell = \tfrac{1}{2}\big((n-2)p^k-m_bp^a\big). \]

\noindent We write $V_{(a,b)}$ whenever $n$ is clear from the context.
\end{definition}

It is immediate from the definition that the dimension of $V_{(a,b)}(n)$ is $m_bp^a-1$. We use the representations $V_{(a,b)}$ and the Mackey functors $\B{i}{j}$ from Definition \ref{B(i,j)} to form a collection of $C_{p^k}$-spectra that we later show are slices. In all that follows, we will let $\nu(b) = \min\{\nu_p(m_b) , k-a \}$ where $\nu_p(m_b)$ is the
  $p$-adic order of $m_b$. We write $\nu$ when $b$ is clear from context.

A proof of the following result may be found in Section 5.

\begin{restatable}{theorem}{SliceV}
\label{Main Slice Theorem}
  For all $1\leq a \leq k$ and $1 \leq b \leq d$, the $C_{p^k}$-spectrum \[S^{V_{(a,b)}} \wedge  H\B{\nu+1}{a-1}\] is an $(m_bp^a-1)$-slice.
\end{restatable}

We will see that the spectra in the theorem above make up all but one of the nontrivial slices in the tower for $S^n \wedge H\mf{\Z}$. The $n$-slice is of a slightly different form and moreover, its definition depends on whether $n$ is a multiple of $p$.

\begin{definition}\label{W and W'}
If $n$ not a multiple of $p$, let $W(n)$ be the $C_{p^k}$-representation 
\[W(n) = V_{(1,1)} + 1 - \bigoplus_{i=1}^{\frac{1}{2}(m_1p-n)}\lambda(i).\]

If $n$ is a multiple of $p$, let $W'(n)$ be the $C_{p^k}$-representation
\[ W'(n) = V_{(1,2)} + 1 - \bigoplus_{i=1}^{p-1}\lambda(i) -
\lambda(1).\]
\end{definition}

We write $W$ and $W'$ in place of $W(n)$ and $W'(n)$ respectively when $n$ is clear from the context. A straightforward computation gives that the dimensions of $W$ and $W'$ are both $n$. 

A proof of the following result may be found in Section 5.

\begin{restatable}{theorem}{SliceW}\label{n-slice 1}
  $\ds S^W \wedge H\mf{\Z}$ and $\ds S^{W'} \wedge H\mf{\Z}$ are both $n$-slices.
\end{restatable}

\subsection{Fiber Sequences}

We will use the slices given in Theorem \ref{Main Slice Theorem} and the maps defined in this section to construct the slice tower of $S^n \wedge H\mf{\Z}$. In the following lemma and proof, we write $S(V, \mf{M})$ in place of $S^V \wedge H\mf{M}$.

\begin{lemma}
\label{Fiber Sequence} 
Let $1 \leq a \leq k$, $1 \leq b \leq d$, and $\widetilde{V}_{(a,b)} = V_{(a,b)} - U$ for any representation
$U$ that consists of a sum of representations of the form $\lambda_r$
for any $r \leq a-1$. Then the $C_{p^k}$-spectrum
$S(V_{(a,b)},\B{\nu+1}{a-1})$ is the fiber of any map of the form
\[ \xymatrix{  S(\widetilde{V}_{(a,b)} +1 +\lambda_{\nu+a},\mf{\Z})
  \ar[r] & S(\widetilde{V}_{(a,b)} +1 + \lambda_{a-1},\mf{\Z}). } \]

\end{lemma}

\begin{proof}
Consider the fiber sequence
\[F \rightarrow H\mf{\Z}(i+j,j) \rightarrow H\mf{\Z}\]
where the map on the right is induced from the map from Definition \ref{B(i,j)}.
We then have the following long exact sequence
\[ 0 \rightarrow \mf{\pi_0(F)} \rightarrow \Z(i+j,j) \rightarrow \Z \rightarrow \mf{\pi_{-1}(F)} \rightarrow 0 \]
Thus by Definition \ref{B(i,j)} we obtain that $F = \Sigma^{-1}H\B{i}{j}$.

 From \cite[Theorem 3.1]{RO(G) slices} we have that
\[S^{\lambda_{i+j} - \lambda_j} \wedge H\mf{\Z} \simeq H\mf{\Z}(i+j,j).\]
Implementing this equivalence and smashing with $S^{1 + \lambda_j}$ yields
\[S^{\lambda_j}H\B{i}{j} \rightarrow S^{\lambda_{i+j}+1} \wedge H\mf{\Z} \rightarrow S^{\lambda_j+1}\wedge H\mf{\Z}.\]
Since $\B{i}{j}$ is $0$ on cells induced up from subgroups of $C_{p^r}$ for all $r \leq j$, the inclusion of the fixed point sphere $S^0 \hookrightarrow S^{\lambda_r}$ induces an equivalence \[H\B{i}{j} \simeq S^{\lambda_r} \wedge H\B{i}{j}\] for all $r \leq j$. Thus we have the sequence
\[ H\B{i}{j} \rightarrow S^{\lambda_{i+j}+1}\wedge H\mf{\Z} \rightarrow S^{\lambda_j+1} \wedge H\mf{\Z} \]
Smashing with $S^{V_{(a,b)}}$ and setting $i = \nu +1$ and $j = a-1$
gives that $S(V_{(a,b)},\B{\nu+1}{a-1})$ is the fiber of 
\[ S(V_{(a,b)}+1+ \lambda_{\nu+a},\mf{\Z}) \rightarrow S(V_{(a,b)} + 1
+ \lambda_{a-1}, \mf{\Z})\]
Lastly, since $H\B{\nu+a}{a-1}$ is equivalent to $S^{\lambda_r} \wedge
H\B{\nu+a}{a-1}$ for all $r \leq a-1$ shifting by representations $U$ consisting solely of such $\lambda_r$
will not change the fiber of the map and the result follows.
\end{proof}

\section{Main Result and Discussion}\label{Main Result}

We can now present the main result, namely, the entire slice tower for $S^n \wedge H\mf{\Z}$ and $G = C_{p^k}$ for odd primes $p$. Note that result below only holds for $n \geq 3$. The cases $n = 1, 2$ are addressed in Section 5.

\begin{maintheorem}\label{Main Result}
  The slice tower for $S^n \wedge H\mf{\Z}$ for $G=C_{p^k}$ for $p$ an odd prime consists of the following data:
  \begin{itemize}
	\item The slice sections $P^m(S^n \wedge H\mf{\Z})$ are of the form $S^V \wedge H\mf{\Z}$ for $n \leq m \leq (n-2)p^k -1$ and the $C_{p^k}$-representations $V$ are determined by the maps from \autoref{Fiber Sequence}.
  \item If $n$ is not divisible by $p$, the $(m_bp^a - 1)$-slices of $S^n \wedge H\mf{\Z}$ are the spectra
    $S^{V_{(a,b)}}\wedge H\B{\nu + 1}{a-1}$ for $1\leq a \leq k$ and
    $1\leq b \leq d$. The $n$-slice is $S^W \wedge H\mf{\Z}$.

  \item If $n$ is not divisible by $p$, the $(m_bp^a - 1)$-slices of $S^n \wedge H\mf{\Z}$ are the spectra
    $S^{V_{(a,b)}}\wedge H\B{\nu + 1}{a-1}$ for $1 \leq b \leq d$ when $2 \leq a \leq k$ and $2\leq b \leq d$ when $a=1$. The $n$-slice is $S^{W'} \wedge H\mf{\Z}$.
\end{itemize}
\noindent All other sections and slices are trivial.
  
\end{maintheorem}

\begin{proof}
  By Theorems \ref{Main Slice Theorem} and \ref{n-slice
    1} we know that the spectra listed in the statement above are slices. We also know that they are fibers
  of the maps from Lemma \ref{Fiber Sequence}. We now show that the
  maps from Lemma \ref{Fiber Sequence} fit together sequentially so that the
  spectra $S^{V_{(a,b)}} \wedge H\B{\nu+1}{a-1}$ are the fibers in the
  appropriate order. In other words, we show that for $1 \leq a \leq
  k$ and $1 \leq b \leq d-1$ there is a map from Lemma \ref{Fiber
    Sequence} with fiber $V_{(a,b+1)}$ that ends in a spectrum
  equivalent to the start of the map with fiber $V_{(a,b)}$. We will
  handle the case $b = d$ separately.

 Recall that $V_{(a,b)} = (n-2)\rho_G - 1
-\bigoplus_{i=1}^\ell\lambda(i)$ where $\ell =
\frac{1}{2}\big((n-2)p^k - m_bp^a\big)$. $V_{a,b+1}$ is defined
similarly but we will write $\ell' =
\frac{1}{2}\big((n-2)p^k-m_{b+1}p^a\big)$. Since $m_{b+1} = m_b+2$
a simple calculation gives that $\ell = \ell'+p^a$. Thus,
\[ V_{(a,b+1)} - V_{(a,b)} =  \bigoplus_{i = \ell' + 1}^{\ell' +
  p^a}\lambda(i) \]

Since $p^a|\ell'$ we know that $\nu_p(i) < p^a$ for all $\ell'+1
\leq j < \ell' + p^a$. Also, $\nu_p(\ell'+p^a) = \nu_p(\ell) = \nu(b)
+ a$. Thus 
\[V_{(a,b+1)} - V_{(a,b)} = U + \lambda_{\nu + a}\]
where $U$ consists solely of representations of the form $\lambda_r$
for $r \leq a-1$.

Now we show that a map from Lemma \ref{Fiber Sequence} with fiber
$V_{(a+1,1)}$ connects exactly to a map with fiber $V_{(a,d)}$. We write $\ell_{a,d}$ and $\ell_{a+1,1}$ in place of
$\ell$ from Definition \ref{V(a,b)} for clarity. More precisely,
\[\ell_{a,d} = \frac{1}{2}\big((n-2)p^k - (n-2)p^a\big)\]
and
\[\ell_{a+1,1} = \frac{1}{2}\big((n-2)p^k - (n-2d)p^{a+1}\big).\]
Thus,
\begin{align*} \ell_{a,d}-\ell_{a+1,1} & = \frac{(n-2d)p^{a+1}}{2} -
  \frac{(n-2)p^a}{2}  = \frac{p^a}{2}\big((n-2d)p - (n-2)\big) \\
& = \frac{p^a}{2}\big((n-(n-\frac{n-n_0}{p} - \delta))p - n + 2\big) = \frac{p^a}{2} (\delta p - n_0 + 2).
\end{align*}

We write $L = \ell_{(a,d)} - \ell_{a+1,1}$. Since $\delta p - n_0 + 2$ is always less than $2p$, we have that
$L < p^{a+1}$. 
By Definition \ref{V(a,b)} we have
\[V_{a+1,1} - V_{a,d} = \bigoplus_{i=\ell_{a+1,1}+1}^{\ell_{a+1,1}+L}
\lambda(i).\]
 Since $p^{a+1}|\ell_{a+1,1}$ and
$L < p^{a+1}$ as shown above, we know that $\nu_p(i) < p^{a+1}$ for
all $\ell_{a+1,1} + 1 \leq i < \ell_{a+1,1}+L$. Also,
$\nu_p(\ell_{a+1,1}+L) = \nu_p(\ell_{a,d}) = \nu(d)+a$. Thus 
\[ V_{a+1,1} - V_{a,d} = U + \lambda_{\nu(d) + a}\]
where $U$ consists solely of representations of the form $\lambda_r$ where $r \leq a$.

Thus, the maps from Lemma \ref{Fiber Sequence} form a tower with
fibers that are the slices listed in the statement of the theorem. Moreover, the slices are in the appropriate order by inspection. All
that is left to show is that the tower in fact begins at the bottom with the $n$-slice and ends at the top with $S^n \wedge H\mf{\Z}$.

When $n$ is not divisible by $p$, it is clear from the definition of $S^W \wedge H\mf{\Z}$ and the fact that $H\B{i}{j}$ is equivalent to $S^{\lambda_r} \wedge H\B{i}{j}$ whenever $r \leq j$ that we have a fiber sequence:
\[ \xymatrix{ S^{V_{(1,1)}}\wedge H\B{\nu(1)+1}{0} \ar[r] &
  S^{\widetilde{V}_{(1,1)} - 1 + \lambda_{\nu(1)+a} } \wedge H\mf{\Z} \ar[d]
  \\  & S^W \wedge H\mf{\Z}} \] where $\widetilde{V}_{(1,1)} = V_{(1,1)} -
\bigoplus_{i=1}^{\frac{1}{2}(m_1p-n)-1}\lambda(i)$. Thus our $(m_1p^1
- 1)$-slice fits in as the next nontrivial slice after the
$n$-slice. Note that $\dim(\widetilde{V}_{(1,1)} - 1 + \lambda_{\nu(1)+a}) = n$. After this stage, we already know each of our maps fit together. Additionally, each map evenly exchanges $\lambda_i$ representations, so we have $S^V \wedge H\mf{\Z}$ at each stage with $\dim(V) = n$. Thus, at the top of the tower we have
\[\xymatrix{ S^{(n-2)\rho_G-1} \wedge H\B{1}{k-1} \ar[r] & S^n \wedge H\mf{\Z} \ar[d] \\ & S^{n-2+\lambda_{k-1}}\wedge H\mf{\Z}}\]
since $V_{(k,d)} = m_d\rho_G-1$ and we must have that
the representation in the $(m_dp^k - 1)$-stage contains $n$ copies of
the trivial representation and be of dimension $n$.

When $n$ is divisible by $p$ we have an identical argument but begin
at the bottom of the tower
with \[ \xymatrix{ S^{V_{(1,2)}}\wedge H\B{\nu(2)+1}{0} \ar[r] & S^{\widetilde{V}_{(1,2)} - 1 +
    \lambda_{\nu(2)+a} } \wedge H\mf{\Z} \ar[d] \\  & S^{W'}
  \wedge H\mf{\Z}} \] where $\ds \widetilde{V}_{(1,2)} = V_{(1,2)} -
\bigoplus_{i=1}^{\frac{1}{2}(m_1p-n)-1}\lambda(i)$.

\end{proof}

\begin{example}
When $G = C_p$ and $n$ is not divisible by $p$ we have the following tower for $S^n \wedge H\mf{\Z}$:

\[ \xymatrix{ 
S^{n-3} \wedge H\B{1}{0} \ar[r] & S^n \wedge H\mf{\Z} \ar[d] \\
S^{n-5} \wedge H \B{1}{0} \ar[r] & S^{(n-2) + \lambda_0} \wedge H\mf{\Z}
\ar[d]\\
S^{n-7} \wedge H \B{1}{0} \ar[r] & S^{(n - 4) + 2\lambda_0} \wedge
H\mf{\Z} \ar[d] \\
& \ar@{..}[d]\\
S^{\lceil \frac{n}{p}\rceil + \epsilon - 1} \wedge H \B{1}{0} \ar[r] &
S^{(n-2d + 2)+ (d-1)\lambda_0} \wedge H\mf{\Z} \ar[d]\\
& S^{(n-2d)+ d\lambda_0} \wedge H\mf{\Z} }\]

where $\lceil \frac{n}{p} \rceil + \epsilon$ has the same parity as $n$.

\end{example}

\begin{example}\label{S7 ex} When $G = C_9$ the slice tower for $S^7 \wedge H\mf{\Z}$ is below. We indicate the slice dimension at the side.

\[ \xymatrix{ 
(5(3)^2 - 1) & S^{5\rho - 1} \wedge H\B{1}{1} \ar[r] & S^7 \wedge H\mf{\Z} \ar[d] \\
(3(3)^2 - 1) & S^{3\rho - 1} \wedge H \B{1}{1} \ar[r] & S^{5 + \lambda_1} \wedge H\mf{\Z}
\ar[d]\\
(5(3) - 1) & S^{2 + \lambda_1} \wedge H \B{1}{0} \ar[r] & S^{3 + 2\lambda_1} \wedge
 H\mf{\Z} \ar[d] \\
(3(3) - 1) & S^{\rho - 1} \wedge H \B{2}{0} \ar[r] &
S^{3 + \lambda_1 + \lambda_0} \wedge H\mf{\Z} \ar[d]\\
&& S^{1 + \lambda_1 + 2\lambda_0} \wedge H\mf{\Z} }\]
\end{example}

\begin{example}\label{S16 ex}
When $G = C_9$ the slice tower for $S^{16} \wedge H\mf{\Z}$ is below. We break the tower in half for ease of arrangement.

\[ \xymatrix{ 
 S^{14\rho - 1} \wedge H\B{1}{1} \ar[r] & S^{16} \wedge H\mf{\Z} \ar[d]  &  
S^{5+4\lambda_1} \wedge H\B{1}{0} \ar[r] & S^{6 + 5\lambda_1} \wedge H\mf{\Z} \ar[d] \\
 S^{12\rho - 1} \wedge H\B{1}{1} \ar[r] & S^{14 + \lambda_1} \wedge H\mf{\Z} \ar[d]  &  
S^{4\rho - 1} \wedge H\B{2}{0} \ar[r] & S^{6+4\lambda_1+\lambda_0} \wedge H\mf{\Z} \ar[d] \\
 S^{10\rho - 1} \wedge H\B{1}{1} \ar[r] & S^{12 + 2\lambda_1} \wedge H\mf{\Z} \ar[d]  &  
S^{3+3\lambda_1} \wedge H\B{1}{0} \ar[r] & S^{4 + 4 \lambda_1 + 2\lambda_0} \wedge H\mf{\Z} \ar[d] \\
 S^{8\rho - 1} \wedge H\B{1}{1} \ar[r] & S^{10 + 3\lambda_1} \wedge H\mf{\Z} \ar[d]  &  
S^{3+2\lambda_1} \wedge H\B{1}{0} \ar[r] & S^{4 + 3\lambda_1+3\lambda_0} \wedge H\mf{\Z} \ar[d] \\
 S^{6\rho - 1} \wedge H\B{1}{1} \ar[r] & S^{8 + 4\lambda_1} \wedge H\mf{\Z} \ar[d] & 
S^{\rho - 1} \wedge H\B{2}{0} \ar[r] & S^{4 + 2\lambda_1 + 4\lambda_0} \wedge H\mf{\Z} \ar[d]\\
&&& S^{2 + 2\lambda_1 + 5\lambda_0} \wedge H\mf{\Z} }\]
\end{example}

In comparing the towers from Examples \ref{S7 ex} and \ref{S16 ex} we see a similarity in pattern. The bottom half of the towers consist of alternating slices formed from $\B{1}{0}$ and $\B{2}{0}$. Both towers begin with $\B{1}{0}$ and end with $\B{2}{0}$. The reason for this is that 16 and 7 differ by the order of our group $|C_9|$. In the tower for $S^{16} \wedge H\mf{\Z}$ the spectra in between form a $3$-adic pattern of $\B{2}{0}$ and $\B{1}{0}$. If we computed the tower for $S^{25} \wedge H\mf{\Z}$ we would see the same start and end and a longer alternating pattern of one $\B{2}{0}$ followed by two $\B{1}{0}$. This pattern, owing to our result relying on the $p$-adic orders of the integers $m_b$, occurs more generally. 

\section{Verification of Slices}
In order to prove Theorems \ref{Main Slice Theorem} and \ref{n-slice 1} we require a result about homology.

\subsection{Preliminaries}
We will need to use precise formulations of restriction and induction of Mackey functors to perform homology computations. The following are given in \cite{RO(G) slices} and are equivalent to Webb's definitions in \cite{Webb}. In each definition below, $G$ is an abelian group with subgroup $H$.

\begin{definition}
 If $\mf{M}$ is any $G$-Mackey functor then the restriction of $\mf{M}$ to $H$, denoted $\big\downarrow^G_H \mf{M}$, is given by 
\[ \big\downarrow^G_H\mf{M} (H/K) = \mf{M}(G/K)\]
for any subgroup $K$ of $H$.
\end{definition}

\begin{definition}  If $\mf{M}$ is an $H$-Mackey functor the induction of $\mf{M}$ up to $G$, denoted $\big\uparrow^G_H \mf{M}$, is given by
\[ \big\uparrow^G_H\mf{M}(G/K) = \mf{M}(i^*_H G/K)\]
for any subgroup $K$ of $G$ where $i^*_H$ is the forgetful functor
from $G$-sets to $H$-sets.

\end{definition}

While the definition of the restricted Mackey functor
$\big\downarrow^G_H\mf{M}$ is obvious for any Mackey functor or group,
it is useful to unpack the definition for
$\big\uparrow^G_H\mf{M}$ further. In our work, we only consider $G$ to
be a cyclic $p$-group and thus the following equivalent way of writing
$\big\uparrow^G_H \mf{M}$ will be useful. It follows immediately from
the definition of Mackey functors and properties of $G$-sets.

\begin{proposition}
If $G$ is a cyclic $p$-group with subgroups $H$ and $K$ and $\mf{M}$
is a $G$-Mackey functor then
\[\big\uparrow^G_H \mf{M} (G/K) = \begin{cases} \Z[G] \otimes_{\Z[H]} \mf{M}(H/K) & K\subseteq H \\
	(\Z[G] \otimes_{\Z[H]} \mf{M}(H/H))^K & H \subseteq K \end{cases} \]
\end{proposition}

This tells us that for $G = C_{p^k}$ the Mackey functor
$\big\uparrow^G_H \mf{M}$ is the fixed point Mackey functor for the
$G/H$-module $\Z[G] \otimes_{\Z[H]} \mf{M}(G/H)$ on subgroups
containing $H$ and is a direct sum of $|G/H|$ copies of the
restriction to $H$ on subgroups contained in $H$.

In performing equivariant computations, we make use of the composition of the restriction and induction operations on Mackey functors.

\begin{proposition}\label{ind res}
Let $G$ be a cyclic $p$-group and $H$ and $K$ be subgroups of $G$. If $\mf{M}$ is a $G$-Mackey functor then the composition of restriction followed by induction is given by
\[\big\uparrow^G_H \big\downarrow^G_H \mf{M} (G/K) = \begin{cases} \Z[G] \otimes_{\Z[H]} \mf{M}(G/K) & K\subseteq H \\
	(\Z[G] \otimes_{\Z[H]} \mf{M}(G/H))^K & H \subseteq K \end{cases} \]
\end{proposition}

We will use these observations to describe the homology of certain representation
spheres taken with coefficients in $\B{i}{j}$. First, we note that if 
\[V = m_k + m_{k-1}\lambda_{k-1} + \cdots m_1 \lambda_1 + m_0 \lambda_0\] a cellular decomposition of $S^V$ arises as a union of cells of the following form in appropriate dimensions.

\begin{proposition}
There is a cellular decomposition
\[S^{\lambda_j} = C_{p^k}/C_{p^{j}+} \wedge e^{1}  \mathop{\cup}_{1-\gamma} C_{p^k}/C_{p^{j}+} \wedge e^{2} \]
where $\gamma$ is a generator of $C_{p^k}$.
\end{proposition}

For details conerning this cellular decomposition, the reader is encouraged to see Section 6 in \cite[Section 6]{HHR Proof} or \cite[Proposition 1.2]{RO(G) slices}. 

\begin{lemma}\label{Hom res} Let $1 \leq j < i \leq k$ where $i+j \leq k$.
If $W$ is a representation of $G=C_{p^k}$ and $C_{p^{i+j}} \subseteq H$ then the restriction map 
\[\mf{H_{-\epsilon}}(S^{-W};\B{i}{j})(G/G) \rightarrow \mf{H_{-\epsilon}}(S^{-W};\B{i}{j})(G/H)\]
is an injection for $\epsilon = 0,1$.
\end{lemma}

\begin{proof}
If $W$ has at least 2 trivial summands, then the homology in dimensions $0$ and $-1$ is trivial so the result holds. If $W$ contains exactly 1 trivial representation then the homology in dimension $0$ is trivial. Furthermore,
\[H_{-1}(S^{-W};\B{i}{j}) = H_0(S^{-W+1}; \B{i}{j}),\]
so it will suffice to consider the case when $W$ has no trivial summands.

Let $G(W)$ denote the maximal proper subgroup of $G$ such that the restriction of $W$ to $G(W)$ has a trivial summand. Then the cellular decomposition of $S^{-W}$ is of the form
\[S^0 \cup G/G(W)_+ \wedge e^{-1} \mathop{\cup}_{1-\gamma} G/G(W)_+
\wedge e^{-2} \cup \cdots\]
where $\gamma$ is a generator of $G$.
The associated chain complex is of the form
\[\xymatrixrowsep{1pc}\xymatrix{\B{i}{j} \ar[r] &
  \big\uparrow_{G(W)}^G \big\downarrow^G_{G(W)}\B{i}{j}
  \ar[r]^{1-\gamma} & \big\uparrow_{G(W)}^G
  \big\downarrow^G_{G(W)}\B{i}{j} \ar[r] & \cdots}\]
where the first map is the restriction on subgroups containing $G(W)$
and is a composition of restriction maps with the
diagonal map on subgroups contained in $G(W)$. From Proposition
\ref{ind res} it follows that our complex on $G$-fixed points is
\[\xymatrixrowsep{1pc}\xymatrix{\B{i}{j}(G/G) \ar[r] &
  \B{i}{j}(G/G(W)) \ar[r] &\B{i}{j}(G/G(W))}.\] 
If $C_{p^{i+j}} \subseteq G(W)$, by Definition \ref{B(i,j)} this
sequence is exact so the statement is trivially true.

If $G(W)$ does not contain $C_{p^{i+j}}$, then $G(W) \subset H$ since
$C_{p^{i+j}} \subseteq H$. Then by Proposition \ref{ind res} the restriction map on homology in dimension $0$ is of the form
\[\xymatrix{\ker(\B{i}{j}(G/G) \rightarrow \B{i}{j}(G/G(W))) \ar[d] \\
            \ker(\B{i}{j}(G/H) \rightarrow  (\Z[G] \otimes_{\Z[G(W)]} \B{i}{j}(G/G(W)))^H ).}\]
Since $C_{p^{i+j}} \subseteq H$, we know that we have an isomorphism given by the composition of restriction maps
\[\B{i}{j}(G/G) \rightarrow \B{i}{j}(G/H).\]
Therefore, the restriction map on $H_0(S^{-W};\B{i}{j})$ is always an injection.

Additionally, since $G(W) \subset H$ we know that 
\[\mf{\ker(1-\gamma)}(G/G) \cong \mf{\ker(1-\gamma)}(G/H) = \B{i}{j}(G/G(W))\]
and since $\B{i}{j}(G/G)$ is also isomorphic to
$\B{i}{j}(G/H)$, we have the desired injection in dimension $-1$ as well.
\end{proof}

\subsection{Slice Results}
To assist in proving Theorems \ref{Main Slice Theorem} and \ref{n-slice 1} we establish a more computationally accesible way to classify $n$-slices.

\begin{definition}
Let $X$ be a $G$-spectrum. If $X \in \tau_{\geq n}$ we write $X \geq n$ and if the localization map $X \rightarrow P^n(X)$ is an equivalence we write $X \leq n$.
\end{definition}

In order to show that a $G$ spectrum $X$ is an
$n$-slice, we must  show that $X \leq n$ and $X \geq n$. We will use the following result to assist in showing $X \geq n$. It is equivalent to \cite[Corollary 3.9]{ESP}
where we let $Y = H\mf{M}$. 

\begin{proposition}
\label{V geq n} Let $W$ be a representation of $G$ and $\mf{M}$ be a
$G$-Mackey functor such that
$S^{W} \wedge H\mf{M} \geq \mbox{dim}(W)$. If $V$ is a subrepresentation of $W$ such that
\begin{enumerate}
\item $V^G = W^G$ and
\item for all proper subgroups $H \subset G$,
the restriction $i^*_H (S^{V} \wedge H\mf{M}) \geq \mbox{dim}(V)$,
\end{enumerate}
then $S^V \wedge H\mf{M} \geq  \mbox{dim}(V)$.
\end{proposition}

We now describe a straightforward way to show that a certain collection of $G$-spectra are slices.

\begin{theorem}\label{SV are slice}Let $G = C_{p^k}$.
Let $V$ be a $G$-representation and $\mf{M}$ be a $G$-Mackey functor. In order to show that the $G$-spectrum $S^{V} \wedge H\mf{M}$ is a dim$(V)$-slice we must show that
\[ V \subset m\rho_G-\epsilon \mbox{\hspace{.5cm} and \hspace{.5cm}} V^G = (m\rho_G - \epsilon)^G \]
for some $m > 0$ and
\[ [S^{-\epsilon}, S^{V-m\rho_G} \wedge H\mf{M}] = 0\]
for all $mp^k - \epsilon > \dim(V)$.
\end{theorem}

\begin{proof}
We induct on the subgroups of $C_{p^k}$. Our underlying spectrum is $S^{\dim(V)} \wedge H\mf{M}$ which is $\geq \dim(V)$ by definition. Since there are no non-trivial maps from spectra of the form $C_{p^k_+} \wedge S^m$ to our spectrum whenever $m > \dim(V)$ we know it is $\leq \dim(V)$ as well.

By our inductive hypothesis and \autoref{V geq n} with $W = m\rho_G - \epsilon$ and our inductive
hypothesis, we need only show that $V \subset m\rho_G - \epsilon$ and $V^G = (m\rho_G-\epsilon)^G$ in order to show that $S^{V} \wedge H\mf{M} \geq \dim(V)$.

To show $S^V \wedge H\mf{M} \leq \dim(V)$, we must
show that 
\[ [G/H_+ \wedge S^{t\rho_H-\epsilon}, S^V \wedge H\mf{M}] = 0 \]
for all $H \subseteq G$ and $t|H| -\epsilon > \mbox{dim}(V)$. By our inductive hypothesis and Spanier-Whitehead duality, it will be
sufficient to show that \[ [S^{-\epsilon},S^{V-t\rho_G}\wedge H\mf{M}] =
0\] for $tp^k -\epsilon > \mbox{dim}(V)$.
\end{proof}

As a first example using the above methodology, we compute the slice tower for $S^n \wedge H\mf{\Z}$ for small $n$.
We know by \cite[Corollay 2.16]{ESP} that $H\mf{\Z}$ is itself a $0$-slice and thus the slice tower is trivial in the case $n=0$. We show below that the towers for $S^1 \wedge H\mf{\Z}$ and $S^2 \wedge H\mf{\Z}$ are also trivial.

\begin{proposition}
  The slice towers for $S^n \wedge H\mf{\Z}$ are trivial for $n=1,2$ when $G= C_{p^k}$.
\end{proposition}

\begin{proof}
It is clear that
  $S^n \wedge H\mf{\Z} \geq n$ for all $n$. By \autoref{SV are slice}, it will suffice to show that 
\[ [S^{-\epsilon}, S^{n-t\rho_G}\wedge H\mf{\Z}]=0 \]
when $n =1,2$ and $t > \frac{n+\epsilon}{p^k}$. The $C_{p^k}$-fixed point part of the chain complex associated to $S^{n-t\rho_G}$ ending in dimension $n-t$ is as follows:
\[\xymatrixrowsep{1pc}\xymatrix{\Z \ar[r]^1  & \Z \ar[r]^0  & \Z \ar[r]^p  & \Z \ar[r]^0  & \cdots  }\]

Since $t > \frac{n+\epsilon}{p} > 0$ and $t$ is an integer, $t \geq 1$. Thus, when $n =1$, $n - t \leq 0$ and when $n = 2$, $n-t \leq 1$. In either case, we can see that the homology in dimensions $0$ and $-1$ are trivial so $S^n \wedge H\mf{\Z} \leq n$ for $n = 1,2$ on $C_{p^k}$
\end{proof}

We are now in a position to show that the spectra we used in the construction of our slice tower are slices.

\SliceV*

\begin{proof}

First, note that if $a = k$, then $V_{(k,b)} = (n-2-2d+2b)\rho_G-1 = m_b\rho_G-1$ so we only need to consider $1 \leq a \leq k-1$. Recall that \[V_{(a,b)} = (n-2)\rho_G - 1 - \bigoplus_{i=1}^{\ell} \lambda(i)\] for $\ell = \frac{1}{2}\big((n-2)p^k - m_bp^a\big)$. Let $L$ be the residue of $\ell$ modulo $p^k$. So we may write $\ell = \ell_kp^k +L$. If $L = 0$, then $V_{(a,b)} = (n-2-2\ell_k)\rho_G-1$ so we are done. Otherwise, $0< L < p^k$ and so \[V_{(a,b)} = (n-2-2\ell_k)\rho_G-1 -\bigoplus_{i=1}^L \lambda(i) \subset (n-2-2\ell_k)\rho_G-1.\]  Furthermore, $\lambda(i)$ is not trivial for $0\leq i \leq L $ and $L < p^k$  so \[(V_{(a,b)})^G = ((n-2-2\ell_k)\rho_G-1)^G.\]

We now show the second part of \autoref{SV are slice}. That is, we show \[[S^{-\epsilon}, S^{V_{(a,b)} - t\rho_G}H\B{\nu+1}{a-1}] = 0\] for $tp^k-\epsilon > \mbox{dim}(V_{(a,b)})$. Again, $\ell=\ell_kp^k + L$ so we can rewrite $V_{(a,b)}-t\rho_G$ as 
\[V_{(a,b)}-t\rho_G = (n-2-2\ell_k-t)\rho_G - 1 - \bigoplus_{i=1}^L \lambda(i).\]
and the condition $tp^k - \epsilon > \mbox{dim}(V_{(a,b)})$ is equivalent to
\[tp^k > (n-2-2\ell_k)p^k - 1 - 2L + \epsilon.\] 

We will now show that $V_{(a,b)}-t\rho_G$ has only non-positive cells for such $t$. We break the proof into two cases depending on the size of $L$.

\underline{Case 1}: $L\leq \frac{p^k-1}{2}$. In this case, $tp^k-\epsilon> \dim(V_{(a,b)})$ is equivalent to
$ tp^k \geq (n-2-2\ell_k)p^k - p^k + \epsilon = (n-3-2\ell_k)p^k + \epsilon $
and hence $t \geq n-2-2\ell_k$. Thus, $V_{(a,b)}-t\rho_G$ is a negative representation.

 \underline{Case 2}: $L > \frac{p^k-1}{2}$. Since we also have that $L<p^k$, $tp^k-\epsilon > \dim(V_{(a,b)})$ is equivalent to
$tp^k  > (n-2-2\ell_k)p^k - 2p^k - 1 +\epsilon = (n-4-2\ell_k)p^k - 1 + \epsilon $
and hence $t \geq n-3-2\ell_k$. Now, we may rewrite $V_{(a,b)} - t\rho_G$ as
\begin{align*}V_{(a,b)}-t\rho_G& = (n-3-2\ell_k-t)\rho_G + \rho_G - 1 - \bigoplus_{i=1}^L \lambda(i)\\
                               & = (n-3-2\ell_k-t)\rho_G - \bigoplus_{i=\frac{p^k+1}{2}}^L \lambda(i). \end{align*}
and so again, we have that $V_{(a,b)}-t\rho_G$ contains only non-positive cells.

Now, since $V_{(a,b)}-t\rho_G$ contains only non-positive cells, by Lemma \ref{Hom res} we have that the restriction
\[H_{-\epsilon}(S^{V_{(a,b)}-t\rho_G};\B{\nu+1}{a-1})(G/G) \rightarrow H_{-\epsilon}(S^{V_{(a,b)}-t\rho_G};\B{\nu+1}{a-1})(G/C_{p^{\nu+a}}),\]
 is an injection. So if we can show that the homology on $G/C_{p^{\nu+a}}$ is trivial, we will have the result.


It is a simple exercise to show that for $H = C_{p^{\nu+a}}$
\[i_H^\ast V_{(a,b)}=\tfrac{m_b}{p^{\nu}}\rho_H-1.\]
Thus, restricting $V_{(a,b)}-t\rho_G$ to $H = C_{p^{\nu+a}}$ yields $(\tfrac{m_b}{p^{\nu}}-t[G:H])\rho_H -1$. Now recall that we are only considering $1 \leq a \leq k-1$ as $a=k$ yields an appropriate slice in a straightforward manner. Thus for $tp^k > \mbox{dim}(V_{(a,b)}) + \epsilon$ we have 
\[H_{-\epsilon}(S^{V_{(a,b)}-t\rho_G};\B{\nu+1}{a-1})(G/C_{p^{\nu+a}}) = 0.\]
\end{proof}

We will use the following in showing that $S^{W} \wedge H\mf{\Z}$ and $S^{W'} \wedge H\mf{\Z}$ are slices.

\begin{lemma}\label{W + pk}
Let $G = C_{p^k}$. Then
  \[W(n) + \rho_G = W(n+p^k)\]
when $n$ is not a multiple of $p$ and 
\[W'(n) + \rho_G = W'(n+p^k)\]
when $n$ is a multiple of $p$.
\end{lemma}

\begin{proof}
We first show that \[V_{(a,b)}(n) + \rho_G = V_{(a,b)}(n+p^{k-a+1}).\]
From Definition \ref{V(a,b)} we have that the left hand side is
equivalent to 
\[ V_{(a,b)}(n) + \rho_G = (n-1)\rho_G - 1 - \bigoplus_{i=1}^{\ell}
\lambda(i)\]
where $\ell = \frac{1}{2}\big( (n-2)p^k - m_bp^a\big)$. Similarly, the
right hand side is
\begin{align*} V_{(a,b)}(n+p^{k-a+1}) & = (n+p^{k-a+1}-2)\rho_G - 1 -
\bigoplus_{i=1}^{\ell'} \lambda(i) \\ & = (n-1)\rho_G -1
+(p^{k-a+1}-1)\rho_G - \bigoplus_{i=1}^{\ell'} \lambda(i)\end{align*}
where $\ell' = \frac{1}{2} \big( (n+p^{k-a+1}-2)p^k - m_b'p^a
\big)$ and $m_b'$ is the number $m_b$ associated to $n + p^{k-a+1}$ rather than n.

Since $2\rho_G = \bigoplus_{i=1}^{p^k} \lambda(i)$ we see that to
show the left and right hand sides are equal, we must show that $\ell'
= \frac{1}{2}(p^{k-a+1}-1)p^k + \ell$. A simple computation gives that $m_b' = m_b+p^{k-a}$
and thus \begin{align*}\ell' &= \frac{1}{2} \big(
(n+p^{k-a+1}-2)p^k-(m_b+p^{k-a})p^a\big) \\ &= \frac{1}{2}\big( (n-2)p^k
- m_bp^a + (p^{k-a+1})p^k -p^k\big) \\ & = \ell + \frac{1}{2}(p^{k-a+1}-1)p^k.\end{align*} 

Now we can prove the lemma:
\[ W(n+p^k) = V_{(1,1)}(n+p^k) + 1 - \bigoplus_{i=1}^{\frac{1}{2}(m_1'p-(n+p^k))}\lambda(i)\]
where $m_1'$ is the number $m_1$ associated to $n + p^k$.
Since $m_1' = m_1 + p^{k-1}$ then $m_1'p-(n+p^k) = m_1p - n$. Since $ V_{(a,b)}(n) + \rho_G = V_{(a,b)}(n+p^{k-a+1})$, the proof for $W(n)$ is complete. $W'(n) + \rho_G = W'(n+p^k)$ is an immediate consequence.
\end{proof}

\SliceW*

\begin{proof}
We first show the statement is true for $S^W \wedge H\mf{\Z}$. So for now we assume $n$ is not a multiple of $p$. By Lemma \ref{W + pk}
we may additionally assume that $2 < n < p^k$. We have
\[W = (n-2)\rho_G - \bigoplus_{i=1}^{\ell} \lambda(i) -
\bigoplus_{i=1}^{\frac{1}{2}(m_1p-n)}\lambda(i) \] where $\ell$ can be
written as $\ell_kp^k + L$ for $L$ the residue of $\ell$ modulo $p^k$. Thus, $W$ may also be written as
\[W = (n-2-2\ell_k)\rho_G - \bigoplus_{i=1}^{L} \lambda(i) - \bigoplus_{i=1}^{\frac{1}{2}(m_1p-n)}\lambda(i). \]
Since $L < p^k$ and $\frac{1}{2}(m_1p-n) < p$ then using $\overline{\rho_G}$ to denote the reduced regular representation of $G$ we have \[\bigoplus_{i=1}^L \lambda(i) \subset 2\overline{\rho_G} \quad \mbox{  and  } \bigoplus_{i=1}^{\frac{1}{2}(m_1p-n)}\lambda(i) \subset \overline{\rho_G}.\] So $W \subset (n-2-2\ell_k)\rho_G$ and $W^G = ((n-2-2l_k)\rho_G)^G$.

Now we must show
\[  [S^{-\epsilon} , S^{W-t\rho_G} \wedge H\mf{\Z}] = 0 \]
when $tp^k - \epsilon > n$. Since $n < p^k$ then when $\epsilon = 0$ we are considering $t \geq 1$ and when $\epsilon = 1$, $t \geq 2$. We have that
\[ W -t\rho_G  = (n-2-t)\rho_G - \bigoplus_{i=1}^\ell \lambda(i) - \bigoplus_{i=1}^{\frac{1}{2}(m_1p-n)}\lambda(i). \]
where $\ell = \frac{1}{2}\big((n-2)p^k - m_1 p\big).$
We break our argument into two cases: $n$ is even and $n$ is odd.

\underline{Case 1}: If $n$ is even, we can write 
\[\ell = \tfrac{1}{2}(n-4)p^k + p^k - \tfrac{1}{2}m_1p = \tfrac{1}{2}(n-4)p^k + L.\] Recall that we use $m_1$ to denote the integer after $\frac{n}{p}$ that is of the same parity as $n$. So, since $n$ is not a multiple of $p$, $\frac{n}{p} < m_1 < \frac{n}{p}+2$. Furthermore, since $n \leq p^k-1$, then $m_1 \leq \frac{p^k-1}{p}+2$. Simplifying 
gives us that $m_1 \leq p^{k-1}+1$. Then we see that we can bound $L = p^k - \frac{1}{2}m_1p$ as follows
\[ \frac{p^k-p}{2} < L < p^k -\frac{n}{2}.\]
In particular, $L>0$ and so we may rewrite $W$ as
\begin{align*} W & = (n-2-(n-4))\rho_G - \bigoplus_{i=1}^L \lambda(i) - \bigoplus_{i=1}^{\frac{1}{2}(m_1p-n)} \lambda(i) \\
                 \label{W eq} & = 2\rho_G - \Big(\bigoplus_{i=1}^L \lambda(i) + \bigoplus_{i=1}^{\frac{1}{2}(m_1p-n)} \lambda(i) \Big)
\end{align*}

By the bounds given for $m_1$, we know that $\frac{1}{2}(m_1p-n) <p$. Also, since $S^{\lambda(i)}$ and $S^{\lambda(j)}$ are equivalent whenever $i$ and $j$ have the same $p$-adic valuation, for our purposes we can rewrite
\begin{align*} \bigoplus_{i=1}^{\frac{1}{2}(m_1p-n)}\lambda(i) = \bigoplus_{i=p^k-\frac{1}{2}(m_1p-n)}^{p^k-1}\lambda(i).\end{align*}
This and the fact that $L = p^k - \frac{1}{2}m_1p$ give us that
\begin{align*} \bigoplus_{i=1}^L \lambda(i) + \bigoplus_{i=1}^{\frac{1}{2}(m_1p-n)} \lambda(i) \subset 2\overline{\rho_G}\end{align*}
Furthermore, since $L > \frac{p^k-p}{2}$ we know that 
\[\frac{p^k-1}{2} - L < \frac{p-1}{2} < p.\]
Thus, since $\overline{\rho_G} = \bigoplus_{i=1}^{\frac{p^k-1}{2}}\lambda(i)$ and 
\[L + \frac{1}{2}(m_1p-n) = p^k - \frac{n}{2} \geq \frac{p^k+1}{2}\]
we have that 
\begin{align}\label{eq2}\overline{\rho_G} \subset \bigoplus_{i=1}^L \lambda(i) + \bigoplus_{i=1}^{\frac{1}{2}(m_1p-n)} \lambda(i).\end{align}

Now we know we may write
\[W-t\rho_G = (2-t)\rho_G  - \big(\bigoplus_{i=1}^L \lambda(i) + \bigoplus_{i=1}^{\frac{1}{2}(m_1p-n)} \lambda(i)\big)\]
and we are considering $tp^k > n + \epsilon$. Since $n < p^k$, we are really just considering $t\geq 1$. If $t > 1$ then clearly $W-t\rho_G$ contains only negative cells. If $t = 1$, $W-t\rho_G$ also contains only negative cells by \ref{eq2}. In either case, considering the chain complex associated to $S^{W-t\rho_G}$ gives the result:
\[\xymatrixrowsep{1pc}\xymatrix{0 \ar[r] & \Z \ar[r]^1 & \Z \ar[r]^0 & \Z \ar[r]^p & \cdots\\ & 2-t & 1-t & -t & }\]
So by Theorem \ref{SV are slice} $S^W \wedge H\mf{\Z}$ is an $n$-slice when $n$ is even.

\underline{Case 2}: If $n$ is odd we write $\ell = \frac{1}{2}(n-3)p^k - \frac{1}{2}(p^k-m_1p)$. Since $n < p^k$, the largest odd $m_1$ is the next odd integer after $\frac{p^k-2}{p}$. So $m_1 \leq p^{k-1}$ and thus, $m_1p \leq p^k$. Since we also have $m_1p > 0$, this gives us that $0 \leq \frac{1}{2}(p^k - m_1p) < \frac{p^k}{2}$. So we may rewrite
\begin{align*} W & = (n-2-(n-3))\rho_G - \bigoplus_{i=1}^L \lambda(i) - \bigoplus_{i=1}^{\frac{1}{2}(m_1p-n)} \lambda(i) \\
                 & = \rho_G - \Big( \bigoplus_{i=1}^L \lambda(i) + \bigoplus_{i=1}^{\frac{1}{2}(m_1p-n)}\Big) \end{align*}
where $L = \frac{1}{2}(p^k - m_1p)$. Since $m_1 < \frac{n}{p}+2$ and $n \geq 3$, we have that
\[L < \frac{p^k-(n+2p)}{2} \leq \frac{p^k - 2p - 3}{2}.\]
Then since $\frac{1}{2}(m_1p-n) < p$ we have that
\[\bigoplus_{i=1}^L \lambda(i) + \bigoplus_{i=1}^{\frac{1}{2}(m_1p-n)} \lambda(i) \subset \overline{\rho_G}.\]
So $W \subset t\rho_G$ whenever $t \geq 1$ completing the proof for $S^W \wedge H\mf{\Z}$.

We now show that
  $S^{W'} \wedge H\mf{\Z}$ is an $n$-slice and assume $n$ is a multiple of $p$.
 By Lemma \ref{W + pk} we again need only consider the case when $2 < n \leq p^k$.

Recall that 
\[W'(n) = V_{(1,2)} + 1 - \bigoplus_{i=1}^{p-1}\lambda(i) -
\lambda(1)\]
or equivalently
\[W'(n) = (n-2)\rho_G - \bigoplus_{i=1}^{\ell} \lambda(i) -  \bigoplus_{i=1}^{p-1}\lambda(i) -
\lambda(1)\]
where now $\ell = \frac{1}{2}\big((n-2)p^k - m_2p\big)$. Since $n$ is divisible by $p$, we can write
\begin{align*} \ell & = \frac{1}{2}(n-2)(p^k - p) + \big(\frac{np-n}{2} - 2\big)p
                  = \frac{1}{2}(n-2)p^k - \frac{n}{2} - p. \end{align*}
We will again consider the cases where $n$ is even and $n$ is odd separately.

\underline{Case 1}: If $n$ is even we can write $\ell = \frac{1}{2}(n-4)p^k + p^k-\frac{n+2p}{2}$. This way we have
\[W' = 2\rho_G - \bigoplus_{i=1}^L \lambda(i) - \bigoplus_{i=1}^{p-1} \lambda(i) - \lambda(1) \]
where $L = p^k - \frac{n+2p}{2}$. Since $n$ is even and a multiple of $p$, then $2 < n \leq p^k - p$ and we can see that we have the following bound on $L$:
\[ \frac{p^k - p}{2} \leq L < p^k - (p+1). \]
By the upper bound, we see that
\[\bigoplus_{i=1}^L \lambda(i) + \bigoplus_{i=1}^{p-1} \lambda(i) + \lambda(1) \subset 2\rho_G - 2\]
so $W' \subset 2\rho_G$ and $W'^G = (2\rho_G)^G$.

By the lower bound on $L$, we see that 
\[\overline{\rho_G} \subset \bigoplus_{i=1}^L \lambda(i) + \bigoplus_{i=1}^{p-1}\lambda(i).\] Thus for all $t > \frac{n+\epsilon}{p^k}$, the $C_{p^k}$-fixed point part of the chain complex associated to $W'-t\rho_G$ is of the form
\[ \xymatrixrowsep{1pc}\xymatrix{0 \ar[r]&\Z \ar[r]^1 & \Z \ar[r]^0 & \Z \ar[r]^p & \cdots \\ & 2-t & 1-t & -t & }\]
Since for $n < p^k$ we are looking at $t\geq 1$, it is easy to see from the complex that
\[ [S^{-\epsilon} , S^{W'-t\rho_G} \wedge H\mf{\Z}] = 0 \]

\underline{Case 2}: When $n$ is odd, we have one extreme case: $n = p^k$. So we will deal with this case separately.

\underline{Case 2a}: $n < p^k$. In this case, we will write $l = \frac{1}{2}(n-3)p^k + \frac{p^k - n - 2p}{2}$ so we can see that 
\[W' = \rho_G - \bigoplus_{i=1}^L \lambda(i) - \bigoplus_{i=1}^{p-1}\lambda(i) - \lambda(1)\]
where $L = \frac{p^k-n-2p}{2}$. Since we've assumed that $n < p^k$, $n$ is odd, and $n$ is divisible by $p$, we know that $p \leq n \leq p^k-  2p$. So we get the following bound on $L$:
\[0 \leq L \leq \frac{p^k - 3p}{2}.\]
By the upper bound we have that $L + (p-1) + 1 \leq \frac{p^k-p}{2}$ so
\[\bigoplus_{i=1}^L \lambda(i) + \bigoplus_{i=1}^{p-1} \lambda(i) + \lambda(1) \subset \overline{\rho_G}.\]
Thus, $W' \subset \rho_G$ and $W'^G = (\rho_G)^G$. Also then $W' \subset t\rho_G$ for $t \geq 1$ and so
\[ [ S^{-\epsilon},S^{W'-t\rho_G}\wedge H\mf{\Z}] = 0. \]
Thus by Theorem \ref{SV are slice} this case is complete.

\underline{Case 2b}: When $n=p^k$ we have
\begin{align*} V_{(1,2)} + 1 - \bigoplus_{i=1}^{p-1}\lambda(i) - \lambda(1) & = (n-2)\rho_G -\bigoplus_{i=1}^{\ell} \lambda(i) - \bigoplus_{i=1}^{p-1}\lambda(i) - \lambda(1) \\
& = (p^k-2)\rho_G - \bigoplus_{i=1}^{\ell+p-1}\lambda(i) - \lambda(1)\end{align*}
since by Definition \ref{V(a,b)}
\begin{align*} \ell &= (p^k-2)\frac{p^k-p}{2} + \big(\frac{p^k-p^{k-1}}{2} - 2 \big)p \\
                    &= p^k\big(\frac{p^k-3}{2}\big) - p. \end{align*}
Thus, 
\[W'(n) = (p^k-2-(p^k-3))\rho_G + 2 - \lambda(1) = \rho_G + 2 -\lambda(1).\]
Now we also know by \cite{ESP} that $H\mf{\Z}^*$ is a $0$-slice and so $S^{\rho_G}\wedge H\mf{\Z}^*$ is a $p^k$-slice. Furthermore, $H\mf{\Z}^* = S^{2-\lambda}\wedge H\mf{\Z}$ and so we see that $S^{W'} \wedge H\mf{\Z}$ is also an $n$-slice in this case.
\end{proof}

\bibliographystyle{amsalpha}

\end{document}